\documentclass{article}[12pt]
\usepackage{t1enc}
\usepackage[utf8]{inputenc}
\usepackage[T1]{fontenc}
\usepackage{amsmath,amssymb,amsfonts,amsthm}
\usepackage{graphicx}
\usepackage[colorlinks=true, allcolors=black]{hyperref}
\usepackage[a4paper, total={17 cm, 26 cm}]{geometry}
\usepackage{mathtools}
\usepackage{parallel,enumitem}
\usepackage{titlesec}
\usepackage{url}

\title{The intersections of typical Besicovitch sets with lines}
\author{Tamás Kátay}
\date{\today}

\newcommand{\defeq}{\vcentcolon=}

\newtheorem{theorem}{Theorem}[section]
\newtheorem{prop}[theorem]{Proposition}
\newtheorem{lemma}[theorem]{Lemma}

\newtheorem{cor}[theorem]{Corollary}

\theoremstyle{definition}
\newtheorem{defi}[theorem]{Definition}
\theoremstyle{definition}

\theoremstyle{definition}

\theoremstyle{definition}

\theoremstyle{definition}

\def\Nat{\ensuremath\mathbb{N}}

\def\Real{\ensuremath\mathbb{R}}

\def\Lines{\ensuremath\mathcal{L}}
\def\Compact{\ensuremath\mathcal{K}}
\def\Compsub{\ensuremath\mathcal{C}}

\begin{document}
	
	\maketitle
	
	\begin{abstract}
		We show that a typical Besicovitch set $B$ has intersections of measure zero with every line not contained in it. Moreover, every line in $B$ intersects the union of all the other lines in $B$ in a set of measure zero.
	\end{abstract}
	
	\section{Introduction}

	A Besicovitch set is a set $B\subseteq\Real^n$ ($n\geq 2$) which contains a unit linesegment in every direction. Besicovitch showed that there exists a Besicovitch set of measure zero in $\Real^2$ (\cite{BES}, see also \cite{FALCONER} Chapter 7). It is easy to see that this gives us a Besicovitch nullset in every dimension $n\geq 2$. Knowing the existence of a Besicovitch nullset it was natural to ask if it is possible to make it even smaller.
	
	\textbf{Kakeya conjecture:} A Besicovitch set in $\Real^n$ necessarily has Hausdorff dimension $n$.
	
	This conjecture is still \textbf{open} except for $n=2$ in which case it turned out to be true (\cite{DAVIES} Davies 1971). The Kakeya conjecture is connected to several famous open questions in various fields of mathematics \cite{TAO}.
	
	Tom Körner proved that if we consider a a well-chosen closed subspace of $\Compact\left(\Real^2\right)$ in which every element contains a unit segment in every direction between $\frac{\pi}{3}$ and $\frac{2\pi}{3}$, then a typical element in this subspace is of measure zero (\cite{KÖRNER} Theorem 2.3). The union of three rotated copies of such a set is a Besicovitch set of measure zero. In this sense it is typical for a Besicovitch set to have measure zero.
	
	There is a variation of the definition of Besicovitch set:
	
	\begin{defi}\label{defi.besicovitchset}
		A \textbf{Besicovitch set} is a set $B\subseteq\Real^n$ ($n\geq 2$) which contains a line in every direction.
	\end{defi}
	
	This gives us a variation of the Kakeya conjecture which is open as well. It is conjectured to be equivalent to the previous form. We will \textbf{work with Definition \ref{defi.besicovitchset}} throughout this paper.
	
	It is clear from Fubini's theorem that if we intersect a Besicovitch nullset with lines of a fixed direction, then almost every intersection is of measure zero. We will use Baire category arguments combined with duality methods to obtain  Besicovitch sets with stronger properties.

	\section{Preliminaries}
	
	\subsection{Dual sets}
	
	We denote the orthogonal projection of the set $A\subseteq\Real^2$ in the direction $v$ by $pr_v(A)$ (where $v$ is a nonzero vector or sometimes just its angle if it leads to no confusion). Similarly $$P_v(A)\defeq\left\{\frac{x-v}{|x-v|}\in S^1:\ x\in A\setminus\{v\}\right\}$$ is the radial projection of $A$ from the point $v$. We may refer to elements of $S^1$ as angles causing no confusion.
	
	\begin{defi}\label{defi.dual}
		Let $l(a,b)$ denote the line which corresponds to the equation $y=ax+b$. We say that $\Lines$ is the \textbf{dual} of $K\subseteq\Real^2$ (or $\Lines$ is coded by $K$) if $\Lines=\{l(a,b):\ (a,b)\in K\}$.
	\end{defi}
	
	A well-known consequence of this definition is the following. For completeness we present the short proof.
	
	\begin{prop}\label{prop.verticalsection}
		Let $K\subseteq\Real^2$ be a set and $\Lines$ its dual. Then the vertical sections of $L\defeq\bigcup\Lines$ are scaled copies of the corresponding orthogonal projections of $K$. More precisely, $L_x=|(x,1)|\cdot pr_{(-1,x)}(K)$.
	\end{prop}
	\begin{proof}
		The vertical section $L_x$ consists of the points of the form $ax+b$ where $(a,b)\in K$. In other words $$L_x=\{ax+b:\ (a,b)\in K\}=\{(x,1)\cdot (a,b):\ (a,b)\in K\}=\left\{|(x,1)|\cdot \frac{(x,1)}{|(x,1)|}\cdot (a,b):\ (a,b)\in K\right\}.$$ And this is exactly the orthogonal projection of $K$ in the direction $(-1,x)$ scaled by the constant $|(x,1)|$.
	\end{proof}
	
	We need to prove a generalization of the previous observation. This generalization will play a key role in the main proof.
	
	\begin{prop}\label{prop.everysection}
		Let $\Lines$ be the dual of the set $K\subseteq\Real^2$, $L\defeq\bigcup\Lines$, and let $e\notin\Lines$ be a line in $\Real^2$. Then the intersection $e\cap L$ is
		
		$(1)$ a scaled copy of an orthogonal projection of $K$ if $e$ is vertical,
		
		$(2)$ otherwise it is the image of $P_{(a_0,b_0)}(K)\setminus\left\{\frac{\pi}{2},\frac{3\pi}{2}\right\}$ by a locally Lipschitz function, where the equation of $e$ is $y=a_0x+b_0$.
	\end{prop}
	\begin{proof}
		(1) is just the previous proposition.
		
		(2): Note that $\Lines$ does not contain vertical lines because it is the dual of $K$. Then $$e\cap L=\{(x,y)\in\Real^2:\ \exists (a,b)\in K\quad y=a_0x+b_0=ax+b\}.$$ So in the intersection $x=\frac{b-b_0}{a_0-a}$ holds (we have $a\neq a_0$ because $e$ does not intersect lines parallel to itself). It is enough to determine the projection of $e\cap L$ to the $x$-axis since $e\cap L$ is the image of this projection by a Lipschitz function.
		
		On the other hand, the projection of $e\cap L$ to the $x$-axis is $\left\{\frac{b-b_0}{a_0-a}:\ (a,b)\in K\right\}=\left\{(-1)\cdot\frac{b-b_0}{a-a_0}:\ (a,b)\in K\right\}$, which is the set of slopes of the lines connecting points of $K$ to $(a_0,b_0)$ multiplied by $(-1)$. It is clear that this set is the image of $P_{(a_0,b_0)}(K)\setminus~\left\{\frac{\pi}{2},\frac{3\pi}{2}\right\}$ by the function $-\tan(\varphi)$ which is locally Lipschitz.
	\end{proof}
	
	We will need the following.
	
	\begin{prop}\label{prop.dualisclosed}
		The union of the dual of a compact set is closed.
	\end{prop}
	
	The proof is an easy exercise, we leave it to the reader.
	
	\subsection{Special code sets}
	
	Let $\lambda$ denote the 1-dimensional Lebesgue measure. For the main proof we need two compact sets with special properties.
	
	The following theorem is due to Michel Talagrand \cite{TALAGRAND}. For a direct proof in English, see \cite{SCIDIR} Appendix A.
	
	\begin{theorem}\label{theorem.talagrand}
		For any non-degenerate rectangle $[a,b]\times [c,d]\subseteq\Real^2$ there exists a compact set $K\subseteq[a,b]\times [c,d]$ such that its projection to the $x$-axis is the whole $[a,b]$ interval, but in every other direction its projection is of measure zero.
	\end{theorem}
	
	\begin{defi}\label{defi.invisibility}
		A set $A\subseteq\Real^2$ is \textbf{invisible} from a point $a\in\Real^2$ if $\lambda(P_a(A))=0$.
	\end{defi}
	
	We will use a theorem of Károly Simon and Boris Solomyak \cite{SIMSOL}:
	
	\begin{theorem}\label{theorem.simonsolomyak}
		Let $\Lambda$ be a self-similar set of Hausdorff dimension 1 in $\Real^2$ satisfying the Open Set Condition, which is not on a line. Then, $\Lambda$ is invisible from every $a\in\Real^2$.
	\end{theorem}
	
	It is an easy exercise to check that the four corner Cantor set of contraction ratio $\frac{1}{4}$ projects orthogonally to an interval in four different directions. It is well-known that this set satisfies the conditions of Theorem \ref{theorem.simonsolomyak}. Rotate it to have an interval as projection to the $x$-axis. Now by an affine transformation we can make it fit to the rectangle $[a,b]\times[c,d]$ while not losing its properties required by Theorem \ref{theorem.simonsolomyak}. By these easy observations we get the following corollary.
	
	\begin{cor}\label{cor.fourcornerinvisible}
		For any non-degenerate rectangle $[a,b]\times [c,d]\subseteq\Real^2$ there exists a compact set $K\subseteq[a,b]\times [c,d]$ such that its projection to the $x$-axis is the whole $[a,b]$ interval, but it is invisible from every point of the plane.
	\end{cor}

	\subsection{Projections of a compact set}
	
	We will need the following two lemmas.
	
	\begin{lemma}\label{lemma.uppersemicont1}
		Let $A$ be a compact set and $f_A: S^1\to\Real$, $f_A(\varphi)=\lambda(pr_\varphi(A))$. Then $f_A$ is upper semicontinuous.
	\end{lemma}
	
	Talagrand proved in \cite{TALAGRAND} that $\left\{f_A:\ A\in \Compact(\Real^2)\right\}$ is the set of non-negative upper semicontinuous functions. We need only the easy direction, hence we present a proof only for that.
	
	\begin{proof}
		Let $c\in\Real$ be arbitrary. We have to verify that $f_A^{-1}((-\infty,c))$ is open. Let $\varphi$ be such that $\lambda(pr_\varphi(A))<c$. Since $pr_\varphi(A)$ is compact as well, it can be covered by finitely many open intervals $I_j$ ($1\leq j\leq l$) for which $\lambda\left(\bigcup_{j=1}^l I_j\right)<c$ holds. This cover shows that $A$ can be covered by rectangles $R_1,\dots,R_l$ whose projections in the direction $\varphi$ are the intervals $I_1,\dots,I_l$. But for the union of finitely many rectangles it is clear that changing $\varphi$ by a suitably small ($<\delta$) angle we can keep the measure of its projection less than $c$. This implies that for any $\varphi'\in (\varphi-\delta,\varphi+\delta)$ we have $$\lambda(pr_{\varphi'}(A))\leq \lambda\left(pr_{\varphi'}\left(\bigcup_{j=1}^l R_j\right)\right)<c.$$ In other words, a neighbourhood of $\varphi$ also lies in $f_A^{-1}((-\infty,c))$, therefore the preimage is open.
	\end{proof}

	\begin{lemma}\label{lemma.uppersemicont2}
		If $A\subseteq\Real^2$ is compact, then $F_A:\Real^2\setminus A\to \Real$, $F_A(v)=\lambda(P_v(A))$ is upper semicontinuous.
	\end{lemma}
	\begin{proof}
		Let $c\in\Real$. We will check that $F_A^{-1}((-\infty,c))$ is open. Let $v$ be a point such that $F_A(v)=\lambda(P_v(A))<c$. Then by compactness we can take a finite cover of $P_v(A)$ by open arcs $I_1,\dots,I_l$ such that $\lambda\left(\bigcup_{j=1}^l I_j\right)<c$. This cover shows that $A$ can be covered by $l$ sectors $R_1,\dots,R_l$ of an annulus such that their radial projections from $v$ are $I_1,\dots,I_l$. For the union of finitely many sectors of an annulus and a point which has a positive distance from them it is clear that moving $v$ by a suitably small distance we can keep the measure of the radial projection of $\bigcup_{j=1}^l R_j$ less than $c$. In other words, a neighbourhood of $v$ lies in $F_A^{-1}((-\infty,c))$, so it is open.
	\end{proof}

	\subsection{Baire category and Hausdorff distance}
	
	For the sake of clarity we assert some well-known definitions and theorems here.

	\begin{defi}\label{defi.firstcategory}
		Let $X$ be a topological space and $E\subseteq X$.
		\begin{itemize}
			\item $E$ is \textbf{nowhere dense} in $X$ if its closure has empty interior.
			\item $E$ is of \textbf{first category} in $X$ if it is the countable union of nowhere dense sets.
			\item $E$ is of \textbf{second} category in $X$ if it is not of first category.
			\item $E$ is \textbf{residual} in $X$ if its complement is of first category.
		\end{itemize}
	\end{defi}

	\begin{theorem}\label{theorem.baire}
		\textbf{(Baire category theorem)} A complete metric space is of second category in itself.
	\end{theorem}

	\begin{defi}\label{defi.typical}
		Let $X$ be a complete metric space. The property $P(x)$ is \textbf{typical} in $X$ if $\{x\in X:\ P(x)\}$ is residual in $X$. We often formulate this in a less accurate manner: a typical $x\in X$ has the property $P(x)$.
	\end{defi}
	
	Let $(X,d)$ be a metric space and let $\Compact(X)$ be the set of its compact subsets. Denote the open $\delta$-neighbourhood of $A$ by $A_\delta$, and denote the closed $\delta$-neighbourhood of $A$ by $\overline{A_\delta}$. 
	
	\begin{defi}\label{defi.hausdorffdistance}
		Let $K,L\in\Compact(X)$. The \textbf{Hausdorff distance} of $K$ and $L$ is $$d_H(K,L)\defeq\max\{\inf\{\delta_1\geq 0:\ K\subseteq L_{\delta_1}\},\inf\{\delta_2\geq 0:\ L\subseteq K_{\delta_2}\}\}.$$
	\end{defi}

	\begin{theorem}\label{theorem.complete}
		If $(X,d)$ is a complete metric space, then $(\Compact(X),d_H)$ is a complete metric space as well.
	\end{theorem}
	
	\section{The main theorem}
	
	We could introduce a new Besicovitch set by simply taking the dual of the compact set given by Corollary \ref{cor.fourcornerinvisible}. It would have intersections of measure zero with every non-vertical line not contained in it by Proposition \ref{prop.everysection}. However, we will go further to obtain the following stronger result:
	
	\begin{theorem}\label{theorem.main}
		There exists a Besicovitch set $B=\bigcup\Lines$ (where $\Lines$ is a family of lines) in the plane such that:
		
		$(1)$ $B$ is closed.
		
		$(2)$ $B$ is of 2-dimensional Lebesgue measure zero.
		
		$(3)$ For every line $e\notin\Lines$ the intersection $B\cap e$ is of 1-dimensional Lebesgue measure zero.
		
		$(4)$ For every $e\in\Lines$ the intersection $e\cap\bigcup(\Lines\setminus\{e\})$ is of 1-dimensional Lebesgue measure zero.
		
		Moreover, we claim that these properties are typical in the sense described below.
	\end{theorem}
	
	We work in $\Compact\left([0,1]^2\right)$ which is a complete metric space with the Hausdorff distance. Consider the subspace $$\Compsub\defeq\left\{K\in\Compact\left([0,1]^2\right):\ pr_{\frac{\pi}{2}}(K)=[0,1]\right\}.$$ It is easy to check that $\Compsub$ is a closed subspace hence a complete metric space as well. The typicality in the main theorem means that a typical $K'\in\Compsub$ codes a family of lines $\Lines'$ for which $L'=\bigcup{\Lines'}$ is an almost Besicovitch set: the union of four rotated copies of $L'$ satisfies all the properties in Theorem \ref{theorem.main}.
	
	The following theorem strengthens Theorem \ref{theorem.talagrand} and it is due to Alan Chang \cite{CHANG}. Here we present our own proof (found independently of Chang) to provide a useful analogue for the proof of the next theorem.
	
	\begin{theorem}\label{theorem.typical1}
		A typical element of $\Compsub$ has orthogonal projections of measure zero in every non-vertical direction.
	\end{theorem}
	
	\begin{proof}
		We have to prove that the set $\{K\in\Compsub:\ \exists\varphi\in [0,\pi]\setminus\{\frac{\pi}{2}\}\quad \lambda(pr_\varphi(K))>0\}$ is of first category. Let $T_n=\left\{\varphi\in[0,\pi]:\ |\varphi-\frac{\pi}{2}|\geq\frac{1}{n}\right\}$. It suffices to show that for every $n$ $$B_n\defeq\left\{K\in\Compsub:\ \exists\varphi\in T_n\quad \lambda(pr_\varphi(K))\geq\frac{1}{n}\right\}$$ is nowhere dense in $\Compsub$.
		
		Fix a compact set $K\in\Compsub$ and $\varepsilon>0$. Denote the open ball of center $A$ and radius $\delta$ by $B_H(A,\delta)$ (with respect to the Hausdorff distance). We need to find $K'\in\Compsub$ and $\varepsilon'>0$ such that $B_H(K',\varepsilon')\subseteq B_H(K,\varepsilon)$ and $B_H(K',\varepsilon')\cap B_n=\emptyset$.
		
		At first we construct $K'$. Take a finite $\frac{\varepsilon}{3}$-net in $K$: $\{(x_1,y_1),\dots,(x_N,y_N)\}$. Consider the squares of the form $$Q_i\defeq\left[x_i-\frac{\varepsilon}{3},x_i+\frac{\varepsilon}{3}\right]\times\left[y_i-\frac{\varepsilon}{3},y_i+\frac{\varepsilon}{3}\right]\qquad (1\leq i\leq N).$$
		Some of the squares may not lie in $[0,1]^2$. We cut off the parts sticking out of $[0,1]^2$ making $Q_i$ a rectangle if it is necessary. Since it was created from an $\frac{\varepsilon}{3}$-net, $\bigcup_{i=1}^N Q_i$ covers $K$. Hence its projection to the $x$-axis is the whole $[0,1]$. For every rectangle $Q_i$ Theorem \ref{theorem.talagrand} gives us a compact set $K_i'\subseteq Q_i$ which has orthogonal projections of measure zero in every non-vertical direction and $pr_{\frac{\pi}{2}}(K_i')=pr_{\frac{\pi}{2}}(Q_i)$. Now let $K'=\bigcup_{i=1}^N K_i'$.
		
		We need to check the following:
		
		(1) $K'\in\Compsub$,
		
		(2) $K'\in B_H(K,\varepsilon)$ and
		
		(3) $\lambda(pr_\varphi(K'))<\frac{1}{n}$ for all $\varphi\in T_n$.
		
		(1) This is clear since $pr_{\frac{\pi}{2}}\left(\bigcup_{i=1}^N Q_i\right)=[0,1]$ and $pr_{\frac{\pi}{2}}(K_i')=pr_{\frac{\pi}{2}}(Q_i)$ in each $Q_i$.
		
		(2) The following two sequences of containments prove that $d_H(K,K')<\varepsilon$. $$K'\subseteq\bigcup_{i=1}^N{Q_i}\subseteq\{(x_1,y_1),\dots,(x_N,y_N)\}_{\frac{2}{3}\varepsilon}\subseteq K_{\frac{2}{3}\varepsilon}$$
		$$K\subseteq\{(x_1,y_1),\dots,(x_N,y_N)\}_{\frac{1}{3}\varepsilon}\subseteq \left(K'_{\frac{\sqrt{2}}{3}\varepsilon}\right)_{\frac{1}{3}\varepsilon}\subseteq K'_{\frac{\sqrt{2}+1}{3}\varepsilon}$$
		
		(3) $K'$ is the union of $N$ sets whose projection is of measure zero in every non-vertical direction.
		
		Now we have to find $\varepsilon'$.
		
		It is very easy to check that for any compact set $A$, positive real number $\delta$ and angle $\varphi$ the following holds: $pr_\varphi(\overline{A_\delta})=\overline{(pr_\varphi(A))_\delta}$.
		
		For every $\varphi$ the projection $pr_\varphi(K')$ is compact, so we have $$\lim_{\delta\to 0}\lambda\left(\overline{(pr_\varphi(K'))_\delta}\right)=\lambda(pr_\varphi(K')).$$ Hence there exists $\varepsilon_\varphi$ for each $\varphi\in T_n$ such that $$\lambda\left(pr_\varphi\left(\overline{K_{\varepsilon_\varphi}'}\right)\right)=\lambda\left(\overline{(pr_\varphi(K')_{\varepsilon_\varphi}}\right)<\frac{1}{n}.$$ The upper semicontinuity ensured by Lemma \ref{lemma.uppersemicont1} for $A=\overline{K_{\varepsilon_\varphi}'}$ says that there exists a $\delta_\varphi$ such that for any $\varphi'\in(\varphi-\delta_\varphi,\varphi+\delta_\varphi)$ the projection is small enough: $\lambda(pr_{\varphi'}\left(\overline{K_{\varepsilon_\varphi}'}\right))<\frac{1}{n}$. On the other hand, $T_n$ is compact, therefore it is covered by finitely many of these neighbourhoods, which gives us finitely many conditions. Hence we can choose $\varepsilon'$ so that $\lambda(pr_{\varphi}(K_{\varepsilon'}'))<\frac{1}{n}$ for all $\varphi\in T_n$. Since every element of $B_H(K',\varepsilon')$ lies in $K_{\varepsilon'}'$, we proved $B_H(K',\varepsilon')\cap B_n=\emptyset$.
		
		If it is necessary, we decrease $\varepsilon'$ further to satisfy $B_H(K',\varepsilon')\subseteq B_H(K,\varepsilon)$.
	\end{proof}
	
	\begin{theorem}\label{theorem.typical2}
		A typical $K\in\Compsub$ is invisible from every point of the plane.
	\end{theorem}
	\begin{proof}
		The proof is very similar to the previous one. We need to prove that $\{K\in\Compsub:\ \exists v\in\Real^2\quad \lambda(P_v(K))>0\}$ is of first category.
		
		First observe that for any point $v\in\Real^2$ and compact set $K\subseteq\Real^2$ $$P_v(K)=\bigcup_{n=1}^\infty P_v\left(K\setminus B\left(v,\frac{1}{n}\right)\right),$$
		which implies $$\lambda(P_v(K))=\lim_{n\to\infty}\lambda\left(P_v\left(K\setminus B\left(v,\frac{1}{n}\right)\right)\right).$$
		Therefore, it suffices to show that $$B_n\defeq \left\{K\in\Compsub:\ \exists v\in[-n,n]\times[-n,n]\quad \lambda\left(P_v\left(K\setminus B\left(v,\frac{1}{n}\right)\right)\right)\geq\frac{1}{n}\right\}$$ is nowhere dense.
		
		Fix  $K\in\Compsub$ and $\varepsilon>0$. Then take a finite $\frac{\varepsilon}{3}$-net $\{(x_1,y_1),\dots,(x_N,y_N)\}$ in $K$ and consider the little squares of side length $\frac{2\varepsilon}{3}$ around them. After chopping off the parts outside $[0,1]^2$ we get the rectangles $Q_1,\dots,Q_N$.
		
		Now for every $Q_i$, Corollary \ref{cor.fourcornerinvisible} gives us a compact set $K_i'\subseteq Q_i$ which is invisible from every point of the plane, and $pr_{\frac{\pi}{2}}(K_i')=pr_{\frac{\pi}{2}}(Q_i)$. Let $K'=\bigcup_{i=1}^N K_i'$. Then $K'$ is also invisible from every point of the plane. Exactly the same argument as in the previous proof shows that $K'\in\Compsub$ and $d_H(K,K')<\varepsilon$ holds.
		
		Now we have to find $\varepsilon'$.
		
		\textbf{Claim.} For every $n\in\Nat$ and $v\in [-n,n]\times[-n,n]$ there exists $\varepsilon_v$ such that $\lambda\left(P_v\left(\overline{K_{\varepsilon_v}'}\setminus B\left(v,\frac{1}{2n}\right)\right)\right)<\frac{1}{n}$.
		
		Fix $n$ and $v$. Restricting the radial projection to an annulus of inner radius $\frac{1}{4n}$ centered at $v$ it becomes a Lipschitz function with Lipschitz constant $4n$. Since $P_v\left(K'\setminus B\left(v,\frac{1}{4n}\right)\right)$ is a compact set of measure zero (recall that even $K'$ is invisible from $v$), we know that $$\lim_{\delta\to 0}\lambda\left(\left(P_v\left(K'\setminus B\left(v,\frac{1}{4n}\right)\right)\right)_\delta\right)=\lambda\left(P_v\left(K'\setminus B\left(v,\frac{1}{4n}\right)\right)\right)=0.$$ Thus for a suitably small $\delta\leq 1$ we have $\lambda\left(P_v\left(K'\setminus B\left(v,\frac{1}{4n}\right)\right)\right)_\delta)<\frac{1}{n}$. Now we claim that $$P_v\left(K_{\frac{\delta}{4n}}'\setminus B\left(v,\frac{1}{2n}\right)\right)\subseteq \left(P_v\left(K'\setminus B\left(v,\frac{1}{4n}\right)\right)\right)_\delta.$$ Indeed, if $x\in K_{\frac{\delta}{4n}}'\setminus B\left(v,\frac{1}{2n}\right)$, then there exists $y\in K'\setminus B\left(v,\frac{1}{4n}\right)$ such that $|x-y|<\frac{\delta}{4n}\leq \frac{1}{4n}$. Therefore $|P_v(x)-P_v(y)|<\delta$ because of the Lipschitz property, and $P_v(y)\in P_v\left(K'\setminus B\left(v,\frac{1}{4n}\right)\right)$, so $P_v(x)\in\left(P_v\left(K'\setminus B\left(v,\frac{1}{4n}\right)\right)\right)_\delta$. Hence $\varepsilon_v=\frac{\delta}{5n}$ is a good choice.
		
		If $\varepsilon_v$ is suitable for $v$, then for every $v'\in B\left(v,\frac{1}{2n}\right)$ $$\overline{K_{\varepsilon_v}'}\setminus B\left(v',\frac{1}{n}\right)\subseteq \overline{K_{\varepsilon_v}'}\setminus B\left(v,\frac{1}{2n}\right)$$ therefore $$\lambda\left(P_{v'}\left(\overline{K_{\varepsilon_v}'}\setminus B\left(v',\frac{1}{n}\right)\right)\right)\leq\lambda\left(P_{v'}\left(\overline{K_{\varepsilon_v}'}\setminus B\left(v,\frac{1}{2n}\right)\right)\right).$$
		For $A=\overline{K_{\varepsilon_v}'}\setminus B\left(v,\frac{1}{2n}\right)$ the function $F_A$ is upper semicontinuous on the complement of $A$ by Lemma \ref{lemma.uppersemicont2}. Hence there exists $U_v\subseteq B\left(v,\frac{1}{2n}\right)$ neighbourhood of $v$ such that for all $v'\in U_v$ $$\lambda\left(P_{v'}\left(\overline{K_{\varepsilon_v}'}\setminus B\left(v',\frac{1}{n}\right)\right)\right)\leq\lambda\left(P_{v'}\left(\overline{K_{\varepsilon_v}'}\setminus B\left(v,\frac{1}{2n}\right)\right)\right)=F_A(v')<\frac{1}{n}.$$
		
		Since $[-n,n]\times[-n,n]$ is compact, it can be covered by finitely many such neighbourhoods, therefore we may choose an $\varepsilon'$ which is suitable for all $v\in [-n,n]\times[-n,n]$.
		
		We need to prove that $B_n\cap B_H(K',\varepsilon')=\emptyset$ holds. Let $L\in B_H(K',\varepsilon')$ and $v\in [-n,n]\times[-n,n]$. Then $L\subseteq K_{\varepsilon'}'$ hence $$\lambda\left(P_v\left(L\setminus B\left(v,\frac{1}{n}\right)\right)\right)\leq\lambda\left(P_v\left(K_{\varepsilon'}'\setminus B\left(v,\frac{1}{n}\right) \right)\right)<\frac{1}{n}$$ by the choice of $\varepsilon'$. Consequently, $L\notin B_n$.
	\end{proof}
	
	Now we have two typical properties in $\Compsub$ by Theorem \ref{theorem.typical1} and Theorem \ref{theorem.typical2}, so we may merge them into one corollary.
	
	\begin{cor}\label{cor.typical}
		A typical element $K\in\Compsub$ has orthogonal projections of measure zero in every non-vertical direction, and it is invisible from every point of the plane.
	\end{cor}
	
	\begin{proof}[Proof of Theorem \ref{theorem.main}]
		
		Let $K'$ be a typical element in $\Compsub$, $\Lines'$ be its dual and $L'\defeq\bigcup\Lines'$. Then $L'$ contains a line of slope $m$ for every $m\in[0,1]$ because the slope is coded by the first coordinate and $pr_{\frac{\pi}{2}}(K')=[0,1]$.
		
		(1) $L'$ is closed by Proposition \ref{prop.dualisclosed}.
		
		(3) Let $e$ be any vertical line. Then its intersection with $L'$ is similar to a non-vertical orthogonal projection of $K'$ by Proposition \ref{prop.everysection}. Therefore, it is of measure zero by Corollary \ref{cor.typical}. This implies (2) immediately.
		
		Now let $e$ be any non-vertical line not in $\Lines'$. Then its intersection with $L'$ is the image of $P_v(K')\setminus\{\frac{\pi}{2},\frac{3\pi}{2}\}$ by a locally Lipschitz function for some point $v\in\Real^2\setminus K'$ (Proposition \ref{prop.everysection} again). Therefore it is of measure zero by Corollary \ref{cor.typical}.
		
		So $L'$ has an intersection of measure zero with every line not contained in it.
		
		(4) Let $e\in\Lines'$ and let $y=a_0x+b_0$ be its equation. Now $\Lines'\setminus\{e\}$ is the dual of $K'\setminus\{(a_0,b_0)\}$, thus the intersection $e\cap\bigcup\left(\Lines'\setminus\{e\}\right)$ is the image of $P_{(a_0,b_0)}(K'\setminus\{(a_0,b_0)\})\setminus\{\frac{\pi}{2},\frac{3\pi}{2}\}$ by a locally Lipschitz function (again Proposition \ref{prop.everysection}). Therefore it is of measure zero by Corollary \ref{cor.typical}. 
		
		Let $B$ be the union of four rotated copies of $L'$. Finally it contains a line in every direction and we have not lost its already checked properties. The proof of the main theorem is complete.
	\end{proof}

	\textbf{Acknowledgement.} I would like to thank Tamás Keleti for the helpful guidence and discussions.


\begin{thebibliography}{10}
		\bibitem{BES} A. S. Besicovitch, \textit{On Kakeya's problem and a similar one}, Math. Z.
		\textbf{27} (1928), 312-20.
		\bibitem{FALCONER} K. J. Falconer, \textit{The geometry of fractal sets}, Cambridge University Press, 1985.
		\bibitem{DAVIES} R. O. Davies, \textit{Some remarks on the Kakeya problem}, Math. Proc. Camb. Philos. Soc. \textbf{69} (2008), 417-421.
		\bibitem{TAO} T. Tao, \textit{From Rotating Needles to Stability of Waves: Emerging Connections between
		Combinatorics, Analysis, and PDE}, Notices Am. Math. Soc. \textbf{48} (2001), 294-303.
		\bibitem{KÖRNER} T. W. Körner, \textit{Besicovitch via Baire}, Stud. Math. \textbf{158} (2003), 65-78.
		\bibitem{TALAGRAND} M. Talagrand, \textit{Sur la mesure de la projection d'un compact et certaines familles de cercles}, Bull. Sci. Math. \textbf{104} (1980), 225–231.
		\bibitem{SCIDIR} W. Hansen, \textit{Littlewood's one-circle problem, revisited}, Expo. Math. \textbf{26} (2008), 365-374.
		\bibitem{SIMSOL} K. Simon, B. Solomyak, \textit{Visibility for self-similar sets of dimension one in the plane}, Real Anal. Exch. \textbf{32} (2007), 67-78.
		\bibitem{CHANG} A. Chang, \textit{Constructions of planar Besicovitch sets}, \url{http://math.uchicago.edu/~ac/alan_chang_topic_proposal_planar_besicovitch.pdf#page7}, unpublished topic proposal (2016).
	\end{thebibliography}
\end{document}